\documentclass[11pt]{amsart}

\usepackage{fullpage}
\usepackage{amssymb, amsmath, wasysym, mathrsfs, mathtools, color, stmaryrd,enumerate}
\usepackage[all, cmtip]{xy}
\usepackage{url}
\usepackage{tikz-cd}

\definecolor{hot}{RGB}{65,105,225}
\usepackage[pagebackref=true,colorlinks=true, linkcolor=hot ,  citecolor=hot, urlcolor=hot]{hyperref}

\theoremstyle{definition}

\theoremstyle{remark}

\numberwithin{equation}{section}

\theoremstyle{break}

\newtheorem{thm}{Theorem}[section]

\theoremstyle{definition}
\newtheorem{rem}{Remark}[section]
\newtheorem*{ackn}{Acknowledgements}

\newtheorem*{thm*}{Theorem}

\newcommand{\Spf}{\operatorname{Specf\!}{}}  
\newcommand{\F}{\mathbb{F}_2}

\newcommand{\GL}{\operatorname{\mathrm{GL}\!}{}}

\begin{document}

\title{A note on the group extension problem to semi-universal deformation}

\author{ An-Khuong DOAN}
\address{IMJ-PRG, UMR 7586, Sorbonne Université, 4 place Jussieu, 75252 Paris Cedex 05, France}
\email{an-khuong.doan@imj-prg.fr}
\address{CNRS-CMLS, UMR 7640, \'{E}cole polytechnique,  91128 Palaiseau Cedex, France.}
\email{ankhuong6147@gmail.com}
\thanks{ }


\date{September 26, 2023.}

\dedicatory{ }


\begin{abstract}
The aim of this note is twofold. Firstly, we explain in detail Remark 4.1 in \cite{doan-a} by showing that the action of the automorphism group of the second Hirzebruch surface $\mathbb{F}_2$ on itself extends to its formal semi-universal deformation only up to the first order. Secondly, we show that for reductive group actions, the locality of the extended actions on the Kuranishi space constructed in \cite{doan-equivariant} is the best one could expect in general.
\end{abstract}

\maketitle

\section{Introduction} Understanding analytic moduli spaces globally is in general an extremely hard task to accomplish even in simple cases. However, one often still would like to study them at least locally. For example, in the case of  the moduli space of complex structures on a given differentiable compact manifold, the best approximation for the local moduli space of a given complex structure $I_0$, provided by the exceptional work of Kodaira-Spencer-Kuranishi (\cite{ks}, \cite{kodaira-complex} and \cite{kuranishi-deformation}) is what's so-called Kuranishi space (or equivalently semi-universal deformation space) - an analytic space containing all complex structures sufficiently close to $I_0$. Heuristically, if the Kuranishi space can be equipped with a compatible action of the stabilizer $G_0$ of $I_0$, then ``its quotient" should be a potential candidate for a neighborhood of $I_0$ or at least should give arise to a morphism with good geometric properties from that quotient to a neighborhood of $I_0$ in the moduli space. This is somehow the motivation for the group extension problem on the Kuranishi spaces done in our previous work \cite{doan-a} and \cite{doan-equivariant}, claiming essentially that it is feasible in general only when the stabilizer is reductive. The goal of this note is to clarify the following two issues arising along the way.

In \cite{doan-a}, we showed that the action of the automorphism group of the second Hirzebruch surface $\F$, denoted by $G$, on $\F$ does not extend to its formally semi-universal deformation $\hat{\mathcal{X}}$. A question naturally posed by the anonymous referee is if the $G$-action extends to the restriction of $\hat{\mathcal{X}}$ on the $n^{\text{th}}$-infinitesimal neighborhood of the base space, which we denote by $\hat{\mathcal{X}}_n$, for small value $n$ (cf. \cite[Remark 4.1]{doan-a}). The answer is obviously yes for $n=1$, which follows immediately from the semi-universality of the deformation considered therein. We even claimed without any proof that the first order is the best one up to which the $G$-action extends. Unfortunately, our somewhat tricky method used there, did not succeed in responding to this question in case that $n\geq 2$. Therefore, in this short note, we shall justify that it is only true in general for $n=1$ by a standard but extremely powerful method in deformation theory -  the power series method, i.e. extending simultaneously the vector fields generated by $G$ on $\F$, together with their Lie bracket relations, order by order (cf. Theorem \ref{t1.6.1} below). The obstruction will reveals itself naturally right in the second order.  Consequently, this example could be served as a vidid illustration for Murphy's law in algebraic geometry. It is nothing but the content of $\S$\ref{se2} and $\S$\ref{se4}. 

In \cite{doan-equivariant}, we proved that actions of reductive complex Lie groups on a given complex compact manifold are locally extendable to its Kuranishi family. In its proof, the globality is lost right at the moment when we pass from compact groups to reductive groups by means of a complexification argument (cf.  \cite[Corollary 4.1 and Theorem 5.2]{doan-equivariant}). $\S$\ref{se5} is dedicated to giving a concrete example which confirms that this loss is somehow reasonable: these extended actions can not be global in general, illustrating the optimality of these results. The main object under deformation is still the second Hirzebruch surface $\F$ but this time instead of extending the action of  the whole automorphism group $G$ of $\F$, we extend only the action of the multiplicative group $\mathbb{C}^*$ considered as a subgroup of $G$. It turns out that this action can not be trivial on the base so that it is not global on the Kuranishi space thought of as a germ of complex spaces around the reference point.

The principle method employed throughout the note is the power series method that we find sufficiently elementary to understand thoroughly obstructions to the group extension problem (or the vector field extension problem) on the (formal) Kuranishi family of $\F$. It is worth remarking that the latter one is a very particular example for the treatment of a much more general problem already dealt by J. Wavrik fifty years ago: deforming cohomology classes (cf. \cite{wav}) of which a useful criterion for vertical extensions of vector fields on $\mathbb{F}_2$ is recalled and used in Remark \ref{wavrem}. 

Through the rest of the paper, by $\F$, $G$, $\pi: \mathcal{X} \rightarrow \mathbb{C}$, $\hat{\pi}: \hat{\mathcal{X}} \rightarrow \Spf(\mathbb{C})$ and $\hat{\mathcal{X}}_n$ we always mean the second Hirzebruch surface, its automorphism group, its semi-universal deformation, its formal semi-universal deformation and the restriction of $\hat{\mathcal{X}}$ on the $n^{\text{th}}$-infinitesimal neighborhood of the base space, respectively. For the other notations and conventions, the reader is referred  to \cite{doan-a}.

\begin{ackn}  We would like to thank the reviewer of \cite{doan-a} for asking the first question and Prof. Julien Grivaux for his careful verification. We are warmly thankful to the referee of this article, whose dedicated work led to a remarkable improvement of it.
\end{ackn}

\section{Vector fields generated by $G$ on $\F$ and $\hat{\mathcal{X}}_1$}\label{se2}
Recall that the second Hirzebruch surface $\F$ is defined to be the projectivization $\mathbb{P}(\mathcal{O}_{\mathbb{P}^1}(2)\oplus \mathcal{O}_{\mathbb{P}^1})$ of the bundle $\mathcal{O}_{\mathbb{P}^1}(2)\oplus \mathcal{O}_{\mathbb{P}^1}$. In terms of algebraic equations, it is given explicitly by $$\F:=\{ ([x:y:z],[u:v])  \in \mathbb{P}^2 \times \mathbb{P}^1 | yv^2=zu^2\} $$ whose semi-universal deformation can be shown to be \begin{equation} \label{kurfa}\mathcal{X}:=\{ ([x:y:z],[u:v],t)  \in \mathbb{P}^2 \times \mathbb{P}^1 \times \mathbb{C} \text{ }|\text{ } yv^2-zu^2 -txuv=0\},
\end{equation} equipped with the projection $\pi$  on the last factor (cf. \cite[Proposition 3.1]{doan-a}).

On the intersection of two standard open sets $U_x=\lbrace x=1 \rbrace $ and $U=\lbrace u=1 \rbrace $ in $\mathbb{P}^1\times \mathbb{P}^1$ (cf. the proof of \cite[Proposition 2.1]{doan-a}), the action of $G:=\text{Aut}(\F)=\mathbb{C}^3 \rtimes\mathrm{GL}(2,\mathbb{C})$ on $\F$ given by 
$$g.(v,[1:y])=\left (\frac{c+dv}{a+bv}, \left [ 1:\frac{y(a+bv)^2}{1+y(a_0v^2+a_1v+a_2)} \right ] \right )$$ where $g= \left((a_0,a_1,a_2)^t, \begin{pmatrix}
a &b \\c 
 &d 
\end{pmatrix}\right) \in G$ and $(v,[1:y])\in \F$. Since the Lie algebra of $G$ is 7-dimensional then we obtain $7$ vector fields on $\F$ $$\begin{cases}
 E_1'=-v^2y^2\partial_y\\ 
 E_2'=-y^2\partial_y\\ 
 E_3'=\partial_v\\ 
 E_4'=-vy^2\partial_y\\ 
 E_5'=2y\partial_y-v\partial_v\\ 
 E_6'=v\partial_v\\ 
 E_7'=2yv\partial_y-v^2\partial_v
\end{cases}$$ on $\F$ with the relations \begin{equation}\label{Lierel}\begin{matrix}
\;\;\begin{cases}
[E_1',E_2']=0 \\ 
 [E_1',E_3']=-2E_4' \\ 
[E_1',E_4']=0  \\ 
[E_1',E_5']=0  \\ 
[E_1',E_6']=-2E_1'\\
[E_1',E_7']=0  
\end{cases}, & \;\;\begin{cases}
[E_2',E_3']=0  \\ 
[E_2',E_4']=0  \\ 
[E_2',E_5']=-2E_2'  \\ 
[E_2',E_6']=0 \\
[E_2',E_7']=-2E_4' 
\end{cases}, & \begin{cases}
[E_3',E_4']=E_2'  \\ 
[E_3',E_5']=-E_3'  \\ 
[E_3',E_6']=E_3\\ 
[E_3',E_7']=E_5'-E_6' 
\end{cases},\\ &\\
\begin{cases}
[E_4',E_5']=-E_4' \\ 
[E_4',E_6']=-E_4' \\ 
[E_4',E_7']=-E_1'
\end{cases}, &\begin{cases}
[E_5',E_6']=0 \\ 
[E_5',E_7']=-E_7' 
\end{cases}, & [E_6',E_7']=E_7'.\;\;\;\;\;\;\;\;\; 
\end{matrix}
\end{equation}
These vector fields correspond to the canonical basis of $\mathbb{C}^3\times \mathrm{M}(2,\mathbb{C})=\mathrm{Lie}(G)$. Moreover, they extend to vector fields on $\hat{\mathcal{X}}_1$ as explained in the introduction. In other words, we also have $7$ vector fields $E_1^{(1)},E_2^{(1)},E_3^{(1)},E_4^{(1)},$ $E_5^{(1)},E_6^{(1)},E_7^{(1)}$ on $\hat{\mathcal{X}}_1$ with the same relations as (\ref{Lierel}) and the restriction of these vectors on the central fiber are nothing but $E_1',E_2',E_3',E_4',E_5',E_6',E_7'$. Therefore, by the discussion in \cite[$\S3.2$]{doan-a}, we can assume that our seven vector fields are of the form (up to the first order with respect to $t$)
\begin{equation}\label{formgen}\begin{cases}
 E_1^{(1)}=\left (-v^2y^2+p_1(v,y)t  \right )\partial_y+q_1(v)t\partial_v+k_1t\partial_t\\ 
 E_2^{(1)}=\left (-y^2+p_2(v,y)t  \right )\partial_y+q_2(v)t\partial_v+k_2t\partial_t\\ 
 E_3^{(1)}=p_3(v,y)t\partial_y+\left (1+q_3(v)t  \right )\partial_v+k_3t\partial_t\\ 
 E_4^{(1)}=\left (-vy^2+p_4(v,y)t  \right )\partial_y+q_4(v)t\partial_v+k_4t\partial_t\\ 
 E_5^{(1)}=\left (2y+p_5(v,y)t  \right )\partial_y+\left (-v+q_5(v)t  \right )\partial_v+k_5t\partial_t\\ 
 E_6^{(1)}=p_6(v,y)t\partial_y+\left (v+q_6(v)t  \right )\partial_v+k_6t\partial_t\\ 
 E_7^{(1)}=\left (2yv+p_7(v,y)t  \right )\partial_y+\left (-v^2+q_7(v)t  \right )\partial_v+k_7t\partial_t
\end{cases}
\end{equation}
where $p_i,q_i$ are polynomials (to be determined) whose degree with respect to each variable does not exceed $2$.  

On the other hand, since $E_i^{(1)}$'s are vectors fields on $\hat{\mathcal{X}}_1$ then by \cite[Lemma 3.1]{doan-a}, we can also write \begin{equation} \label{formver}E_i^{(1)}=g_i(v,t)\frac{\partial }{\partial v}+(\alpha_i(v,t)y^2+\beta_i(v,t)y+\gamma_i(v,t))\frac{\partial }{\partial y}+k(t)\frac{\partial }{\partial t}\;(\mod t^2)\end{equation} where $$\begin{cases}
 g_i(v,t)=A_i(t)v^2+B_i(t)v+C_i(t)\\ 
 \alpha_i(v,t)=a_i(t)v^2+b_i(t)v+c_i(t)\\ 
 \beta_i(v,t)=-2(a_i(t)t+A_i(t))v+e_i(t)\\ 
\gamma_i(v,t)=tA_i(t)
\end{cases} (\mod t^2).$$ Here $A_i(t),B_i(t),C_i(t),a_i(t),b_i(t),c_i(t),e_i(t)$ are linear functions in $t$ with a relation 
\begin{equation} \label{obs}e_i(t)t+B_i(t)t-k_i(t)=0\;(\mod t^2) \end{equation} for $i=1,2.$ 

As the restriction of $E_i^{(1)}$'s on the central fiber are $E_i'$'s, respectively, direct computation shows that
$$\begin{cases}
E_1^{(1)}= \{ -v^2y^2+t[( a_1v^2+b_1v+c_1) y^2+(2v(1-A_1)+e_1)y] \}  \partial_y+(A_1v^2+B_1v+C_1)t\partial_v\\ 
 E_2^{(1)}= \{ -y^2+t[( a_2v^2+b_2v+c_2) y^2+(-2vA_2+e_2)y] \}  \partial_y+(A_2v^2+B_2v+C_2)t\partial_v
\end{cases}$$  where  $a_i,b_i,c_i, A_i,B_i,C_i,e_i$  are some complex numbers ($i=1,2$).   As a matter of fact, in view of (\ref{formgen}), an easy identification provides 
\begin{equation}\label{equivsys}\begin{cases}
p_1(v,y)= ( a_1v^2+b_1v+c_1) y^2+(2v(1-A_1)+e_1)y \\ 
q_1(v)=A_1v^2+B_1v+C_1 \\
p_2(v,y)= ( a_2v^2+b_2v+c_2) y^2+(-2vA_2+e_2)y \\ 
q_2(v)=A_2v^2+B_2v+C_2.
\end{cases}
\end{equation}
\section{Extensions of vector fields on $\hat{\mathcal{X}}_{1}$ to $\hat{\mathcal{X}}_{2}$. } \label{se4}
We would like to give an explicit description of $E_i^{(1)}$ ($i=1,\cdots,7$) by solving the system of equations given by the Lie relations between $E_i^{(1)}$'s. An important remark is in order.
\begin{rem} Note that there is no component $c\frac{\partial}{\partial t }$ in all $E_i^{(1)}$ ($c$ is a complex number). So, the Lie bracket of $E_i^{(1)}$'s is well-defined when we take modulo $t^2$.
\end{rem}

Now, we wish to find $q_1$, $q_2$ and $q_4$ given in (\ref{formgen}). By computing explicitly the Lie brackets $[E_1^{(1)},E_6^{(1)}], [E_3^{(1)},E_4^{(1)}],[E_3^{(1)},E_5^{(1)}], [E_1^{(1)},E_3^{(1)}] ,[E_5^{(1)},E_7^{(1)}] ,[E_1^{(1)},E_2^{(1)}] ,[E_1^{(1)},E_7^{(1)}]$ using their general form (\ref{formgen}) and then taking into account the Lie relations (\ref{Lierel}), we obtain $k_1=k_2=k_3=k_4=k_7=0$ and 
 \begin{equation}\label{sys}\left\{\begin{matrix}
\partial_vq_4+k_3q_4-k_4q_3-q_2 &=&0\; \\ 
-\partial_vq_1+2q_4&=&0\;\\ 
  -2vq_1+v^2\partial_vq_1&=&0\;\\
 -v^2y^2\partial_yp_2-2yp_1+y^2\partial_yp_1+2v^2yp_2+2vy^2q_2&=&0.
\end{matrix}\right.
 \end{equation}
 From the first three equations of (\ref{sys}) and (\ref{equivsys}), it follows that \begin{equation}\label{exform}\begin{cases}
q_1=A_1 v^2  \\q_4=A_1 v \\ 
 q_2=A_1. 
\end{cases}\end{equation} Substituting (\ref{exform}) into the last equation of (\ref{sys}) gives \begin{align*}
0 &= -v^2y^2\partial_yp_2-2yp_1+y^2\partial_yp_1+2v^2yp_2+2vy^2q_2\\ 
 &=v^2y(2p_2-y\partial_yp_2)-y(2p_1-y\partial_yp_1)+2vy^2A_1 (\text{ since }q_2=A_1) \\ 
 &= v^2y(-2vA_2+e_2)y-y(2v(1-A_1)+e_1)y+2vy^2A_1 \\ 
 &= y^2\left ( v^2(-2vA_2+e_2)- (2v(1-A_1)+e_1)+2A_1v\right )\\
&=y^2\left ( -2v^3A_2+e_2v^2- 2v(1-2A_1)-e_1\right )
\end{align*} which implies that $ A_2=0,  e_1=e_2=0, A_1=\frac{1}{2} $.
 Hence, $q_1=\frac{1}{2}v^2, q_2=\frac{1}{2}, q_4=\frac{1}{2}v.$
Thus, we have \begin{equation}\label{eqfinal}\begin{cases}
   E_1^{(1)}=\left (-v^2y^2+p_1(v,y)t  \right )\partial_y+\frac{1}{2}tv^2\partial_v\\ 
  E_2^{(1)}=\left (-y^2+p_2(v,y)t  \right )\partial_y+\frac{1}{2}t\partial_v
\end{cases}(\mod t^2)\end{equation}
up to the first order.
\begin{thm}\label{t1.6.1}
The action of $G$ on $\F$ extends to the formally semi-universal deformation $\hat{\mathcal{X}}$ only up to the first order.
\end{thm}
\begin{proof}
We shall prove it by contradiction. Suppose that the vector fields $E_1^{(1)}, E_2^{(1)}$ in (\ref{eqfinal}) could be extended to the vector fields $E_1^{(2)}, E_2^{(2)}$ of the second order. In other words, there would exist $m_i(v,y), n_i(v),l_i$ ($i=1,2$) such that 
$$\begin{cases}
   E_1^{(2)}=\left [-v^2y^2+p_1(v,y)t+m_1(v,y)t^2   \right ] \partial_y+\left [\frac{1}{2}tv^2+n_1(v)t^2  \right ]\partial_y+l_1t^2\partial_t\\ 
 E_2^{(2)}=\left [-y^2+p_2(v,y)t +m_2(v,y)t^2  \right ] \partial_y+\left [\frac{1}{2}t +n_2(v)t^2 \right ]\partial_y+l_2t^2\partial_t
\end{cases}$$
and the relation 
\begin{equation}\label{mainobs}[E_1^{(2)},E_2^{(2)}]=0\;(\mod t^3) \end{equation} holds true where $m_i,n_i$ are some polynomials in variables $v,y$ whose degrees with respect to $v$ and $y$ do not exceed $2$ and $l_i$'s are complex numbers. However, by direct inspection, we have that the coefficient of the $\partial_v$-component of $[E_1^{(2)},E_2^{(2)}]$ is 
$$-\frac{1}{2}t^2v+\frac{l_1}{2}t^2-\frac{l_2}{2}t^2v^2 \;(\mod t^3).$$ This quantity is never zero $(\mod t^3)$ for any choice of $l_1$ and $l_2$. Thus, the Lie bracket relation (\ref{mainobs}) is not preserved when we extend them to the second order, which is a contradiction. 
\end{proof}

\begin{rem} Although we can not extend the vector fields $E_1',\ldots,E_7'$ to the semi-universal deformation of $\mathbb{F}_2$, preserving the Lie relation (\ref{Lierel}), an interesting question is if there exist extensions of $E_1',\ldots,E_7'$ on the Kuranishi family, not necessarily subject to (\ref{Lierel}) so that they generate a finite-dimensional Lie algebra. In that case, which Lie groups are associated to it? We thank the anonymous referee again for posing this question and hope to be able to address it in a future work.
\end{rem}
\section{Convergent equivariant Kuranishi family of $\mathbb{F}_2$} \label{se5}

In this final section, we show that the local actions constructed in \cite[Theorem 5.2]{doan-equivariant} can not be global in general by giving a detailed example which is still about the Hirzebruch surface $\F$ and its Kuranishi family $\pi:\; \mathcal{X}\rightarrow (\mathbb{C},0)$ considered in \cite{doan-a}. This time, we consider further an action of $G:=\mathbb{C}^*\times \mathbb{C}^*$ on $\F$ as follows. $\mathbb{C}^*\times \mathbb{C}^*$ can be embedded into $\GL(2,\mathbb{C})$ as a subgroup of invertible diagonal matrices. In particular, its action on $\mathbb{F}_2$ given by  
$$g(p)=\begin{cases}
\left (\left [xu^2:ya^2u^2:yd^2v^2  \right ] ,\left [au:dv  \right ]  \right )& \text{ if } u \not= 0 \\ 
 \left (\left [xv^2:za^2u^2:zd^2v^2  \right ] ,\left [au:dv  \right ]  \right )& \text{ if } v\not=0 
\end{cases},$$ or equivalently
$$g(p)= \left (\left [x:ya^2:zd^2  \right ] ,\left [au:dv  \right ]  \right ),$$ where $p= \left (\left [x:y:y  \right ] ,\left [u:v  \right ]  \right ) \in \F$.

An application of \cite[Theorem 5.2]{doan-equivariant} gives us a local $G$-equivariant structure on $\pi:\; \mathcal{X} \rightarrow \mathbb{C}$. We shall first show that if there is any extended $G$-action  on $\mathcal{X}$, then $G$ can not act trivially on the Kuranishi space. Indeed, by the same analysis as in \S\ref{se2}, the $G$-action on $\mathbb{F}_2$ this time gives rise to two commuting vector fields on $\mathbb{F}_2$ $$\begin{cases}
 E_5'=2y\partial_y-v\partial_v\\ 
 E_6'=v\partial_v.\\ 
\end{cases}$$ Suppose for the moment that the $G$-action on $\F=\mathcal{X}_0:=\pi^{-1}(0)$ could be lifted to a $G$-action on $\mathcal{X}$, then we have two commuting vector fields on $\mathcal{X}$, say $E_5$ and $E_6$.  As before, we can assume that our two vector fields $E_5$ and $E_6$ are of the form (\ref{formver}). Once again, the restriction of these above vector fields on the central fiber are nothing but $E_5',E_6'$ respectively. This implies that
$$\begin{cases}
E_5=\left [2y+t\big((a_5v^2+b_5v+c_5)y^2+(-2A_5v+e_5)y \big)  \right ]\partial_y+\left [ -v+t\left ( A_5v^2+B_5v+C_5 \right ) \right ]\partial_v+k_5t\partial_t\\ 
 E_6=t\left [(a_6v^2+b_6v+c_6)y^2+(-2A_6v+e_6)y   \right ]\partial_y+\left [ v+t\left ( A_6v^2+B_6v+C_6 \right ) \right ]\partial_v+k_6t\partial_t.
\end{cases}$$ up to the first order.
Now, the obstruction of lifting vector fields (\ref{obs}) forces $k_5=k_6=1$. It follows that $E_5$ and $E_6$ are not vertical. Therefore, if the $G$-action was extended then, $G$ (more precisely, two subgroups $\mathbb{C}^*\times \lbrace 1 \rbrace $ and $\lbrace 1 \rbrace\times \mathbb{C}^*$) could not act trivially on the base.
\begin{rem} \label{wavrem} A quicker way to see the non-verticality of $E_5$ and $E_6$, pointed out to us by the anonymous referee is to use Wavrik's criterion. Namely, in \cite[\S5]{wav}, he gave a necessary condition for a general vector field $$\theta:= (Av^2+Bv+C)\partial_v +\left[(av^2+bv+c)y^2 +(-2Av+e)y \right]\partial_y$$ on $\mathbb{F}_2$ to be vertically extended to its Kuranishi family:  \begin{equation}\label{wavobs}B+e=0\end{equation}which is not satisfied for $E_5'$ and $E_6'$ due to the fact that $B_5+e_5=-1+2=1$ and $B_6+e_5=1+0=1$. Hence, $E_5$ and $E_6$ can not be vertical. Taking advantage of this occasion, we emphasize that (\ref{wavobs}) is not sufficient. A necessary and sufficient condition is obtained by letting $k_1(t)\equiv 0$ in \cite[Lemma 3.1(3.6)]{doan-a}). Thus, our obstruction therein generalizes Wavrik's one (\ref{wavobs}).
\end{rem}
Finally, recall that the family (\ref{kurfa}) is semi-universal only for $t=0$ and complete for $t\neq 0$. Moreover, its germ of deformation $(\mathcal{X},\pi^{-1}(0))\rightarrow (\mathbb{C},0)$ is unique. A possible $G$-action on $\mathcal{X}$ that extends the initial $G$-action on $\mathbb{F}_2$ is given explicitly by 
\begin{equation}\label{act}g(p)= \left (\left [x:ya^2:zd^2  \right ] ,\left [au:dv  \right ],adt  \right )\end{equation} for $p:=([x:y:z],[u:v],t) \in \mathcal{X}$ and $g=(a,d)\in \mathbb{C}^*\times \mathbb{C}^* $. For $\mathbb{C}^*$-action, we can restrict the $\mathbb{C}^*\times \mathbb{C}^*$-action on its subgroup $\lbrace (\alpha, 1) \mid \alpha \in \mathbb{C}^*\rbrace.$ Evidently, this $\mathbb{C}^*$-action can not be global on the germ of complex space $(\mathbb{C},0)$.
\begin{rem} Of course, if we consider the entire complex line $\mathbb{C}$ as the Kuranishi space instead of only the germ of complex spaces $(\mathbb{C},0)$, then   the family (\ref{kurfa}) as well as the $\mathbb{C}^*\times \mathbb{C}^*$-action (\ref{act}) are still well-defined. This indicates somehow that our example is still not completely convincing.

\end{rem}

\end{document}